\documentclass{article}
\usepackage{amsfonts}
\usepackage{latexsym}
\usepackage{mathrsfs}
\usepackage{amssymb}
\usepackage{amsmath}
\usepackage{amsthm}
\usepackage{indentfirst}
\usepackage{color}
\usepackage{float}
\usepackage{graphicx}
\usepackage{cite}

\allowdisplaybreaks
\newtheorem{theorem}{\color{black}\indent Theorem}[section]
\newtheorem{lemma}{\color{black}\indent Lemma}[section]

\newtheorem{definition}{\color{black}\indent Definition}[section]
\newtheorem{remark}{\color{black}\indent Remark}[section]

\begin{document}
\title{\LARGE\bf Initial boundary value problem for a strongly damped wave equation with a general nonlinearity}
\author{Hui Yang\qquad Yuzhu Han $^{\ddag}$}
 \date{}
 \maketitle

\footnotetext{\hspace{-1.9mm}$^\ddag$Corresponding author.\\
Email addresses: yzhan@jlu.edu.cn(Yuzhu Han).

\thanks{
$^*$Supported by NSFC (11401252) and
by The Education Department of Jilin Province (JJKH20190018KJ).}}
\begin{center}
{\noindent\it\small School of Mathematics, Jilin University,
 Changchun 130012, P.R. China}
\end{center}

\date{}
\maketitle

{\bf Abstract}\ In this paper, a strongly damped semilinear wave equation with a general nonlinearity is considered.
With the help of a newly constructed auxiliary functional and the concavity argument,
a general finite time blow-up criterion is established for this problem.
Furthermore, the lifespan of the weak solution is estimated from both above and below.
This partially extends some results obtained in recent literatures and
sheds some light on the similar effect of power type nonlinearity and logarithmic nonlinearity on
finite time blow-up of solutions to such problems.



\section{Introduction}
\setcounter{equation}{0}

In this paper, we consider the following initial-boundary value problem for
a damped semilinear wave equation with a general nonlinearity

\begin{equation}\label{eq}
\begin{cases}
u_{tt}-\Delta u-\Delta u_t+ u_t=f(u), & x\in\Omega,t>0,\\
u(x,t)=0, & x\in\partial\Omega,t> 0,\\
u(x,0)=u_0(x), \quad u_t(x,0)=u_1(x), & x\in \Omega,
\end{cases}
\end{equation}
where $\Omega\subset \mathbb{R}^n(n\geq1)$ is a bounded domain with smooth boundary $\partial \Omega$,
$u_0\in H_0^1(\Omega)$, $u_1\in L^2(\Omega)$. The nonlinearity $f\in C^1(\mathbb{R})$ is supposed to satisfy the following assumptions:

(H1) There exists a constant $p>2$ such that
\begin{eqnarray*}
sf(s)\geq p F(s),\ \forall\ s\in \mathbb{R},
\end{eqnarray*}
where $F(s)=\int_0^s f(t) {\rm d}t$;

(H2) There exist positive constants $\alpha, \beta$ and $q\in(1,2^\ast-1)$ such that
\begin{eqnarray*}
|f(s)|\leq \alpha+\beta|s|^{q},\ \forall\ s\in \mathbb{R},
\end{eqnarray*}
where $2^\ast$ is the Sobolev conjugate of $2$, i.e., $2^\ast=+\infty$ for $n=1,2$ and $2^\ast=\frac{2n}{n-2}$ for $n\geq3$;

(H3) There exist positive constants $k_0,k_1$ such that
\begin{eqnarray*}
|f'(s)|\leq k_0+k_1|s|^{l_1},\ \forall\ s\in \mathbb{R},
\end{eqnarray*}
where $ l_1\in (0, 2^\ast-2)$.

Due to the wide applications to physics and to other applied sciences,
considerable attentions have been paid to the study of the qualitative
properties of solutions to semilinear wave equations.
For a quick start, we only refer the interested readers to
\cite{Bociu,Chen,Cholewa,Esquivel-Avila,Gazzola,Georgiev,Gerbi,Guo,Hamza,
Ikehata,Levine1974,Levine2001,Liu,Pata,Song,Sun,Webb,Zhou}
and the references therein for the backgrounds
and the motivations to the study of problem \eqref{eq}.
Among them much effort have been devoted to the study of wave equations
with damping terms, which include both strong damping and linear weak damping
\cite{Cholewa, Esquivel-Avila, Gazzola, Georgiev, Gerbi, Guo, Zhou}.
For instance, Gazzola and Squassina \cite{Gazzola}
considered the following initial boundary value problem for a damped
semilinear hyperbolic equation with power type nonlinearity
\begin{equation}\label{eg1}
\begin{cases}
u_{tt}-\Delta u-\omega \Delta u_t+\mu u_t=|u|^{p-2}u, & x\in\Omega,t>0,\\
u(x,t)=0, & x\in\partial\Omega,t> 0,\\
u(x,0)=u_0(x), \quad u_t(x,0)=u_1(x), & x\in \Omega,
\end{cases}
\end{equation}
where $\omega\geq0$, $\mu>-\omega\lambda_1$ with $\lambda_1$ being the
first eigenvalue of the operator $-\Delta$ in $\Omega$ under homogeneous Dirichlet boundary condition.
Among many other interesting results, they obtained the existence and nonexistence of global solutions and their asymptotic behaviors for initial data at subcritical and critical initial energy levels, respectively.
Moreover, when $\omega=0, \mu\geq0$, they proved that the solutions to problem \eqref{eg1} blow up in finite time for supercritical initial energy, leaving the problem open whether problem \eqref{eg1} with $\omega>0$ admits finite time blow-up solutions when the initial energy is supercritical. Recently,
Yang and Xu \cite{Yang} gave a positive answer to this open problem,
by constructing a new auxiliary functional.

There are also some works dealing with damped semilinear wave equations with logarithmic nonlinearities.
For example, Di et al. \cite{Di} considered the following initial boundary value problem for a strongly damped semilinear wave equation with logarithmic nonlinearity
\begin{equation}\label{eg2}
\begin{cases}
u_{tt}-\Delta u-\Delta u_t=|u|^{p-2}u \ln|u|, & x\in\Omega,t>0,\\
u(x,t)=0, & x\in\partial\Omega,t> 0,\\
u(x,0)=u_0(x), \quad u_t(x,0)=u_1(x), & x\in \Omega.
\end{cases}
\end{equation}
By using the classical potential well method, they gave the threshold
results for the solutions to exist globally or to blow up in finite time,
when the initial energy is subcritical and critical, respectively.
Meanwhile, they derived the bounds for the blow-up time from both above and below.
However, the authors leave two interesting problems open.

(I) As we know, the classical potential well method is applicable only to the case that
the initial energy is subcritical, when it guarantees the invariance of the unstable set $\mathcal{N}_{-}$
(see its definition in Section 2).
Therefore, the sufficient conditions for finite time blow-up for problem \eqref{eg2} given in \cite{Di}
require that the initial energy is subcritical. Naturally, the first problem
is that what will happen to the solutions to problem \eqref{eg2} when the initial energy is supercritical.

(II) The blow-up results obtained in \cite{Di} hold for all $p\in(2,2^*)$,
but the lower bound for the blow-up time was derived only when $p$ is subcritical,
i.e., $2<p<\frac{2n-2}{n-2}$ for $n\geq 3$.
The main reason for this restriction comes from the Sobolev embedding $H_0^1(\Omega)\hookrightarrow L^{2p-2}(\Omega)$, which does not hold for the supercritical exponent $p\in(\frac{2n-2}{n-2},\frac{2n}{n-2})$.
Then, the second problem is whether a lower bound for the blow-up time can be obtained for supercritical exponent $p$.
Recently, Zu and Guo \cite{Zu} obtained a new blow-up criterion for problem \eqref{eg2}
which contains the case that the initial energy is supercritical.
Moreover, when $p\in(\frac{2n-2}{n-2},\frac{2n}{n-2})$,
they also derived a lower bound for blow-up time with the help of a first order differential inequality for a newly constructed auxiliary functional.

By comparing the results of \cite{Gazzola,Yang} with that of \cite{Di,Zu}, we find that power type nonlinearity and logarithmic nonlinearity have similar effect on finite time blow-up of solutions to initial-boundary value problems for damped semilinear wave equations, which is the main purpose of this paper.
More precisely, we shall consider the blow-up property of solutions to problem \eqref{eq} with a general nonlinearity.
By borrowing some ideas from \cite{Yang} and constructing a new auxiliary functional,
we show that the unstable set $\mathcal{N}_{-}$ is invariant under the semi-flow of \eqref{eq}.
With the help of this and Levine's concavity argument, 
we give a new finite time blow-up criterion for problem \eqref{eq} with initial energy
bounded from above by $C_0\big(\|u_0\|^2+2(u_0,u_1)\big)$ (see \eqref{E(0)} in Section 3).
This in particular implies that problem \eqref{eq} admits finite time blow-up solution with 
initial energy at arbitrarily high level, at least for the power-type nonlinearity (see Remark \ref{rem2}).
An upper bound for the blow-up time is derived at the same time.
Finally, by making full use of the (strong) damping term and choosing appropriate parameters in H\"{o}lder's inequality,
we also give an estimate of the lower bound for the blow-up time.

It is easy to verify that both $f(s)=|s|^{p-2}s$ and $f(s)=|s|^{p-2}s\ln|s|$ with $p\in(2,2^*)$ satisfy the assumptions (H1)-(H3).
In this way, it is no more necessary to distinguish the cases $f(u)=|u|^{p-2}u$ and $f(u)=|u|^{p-2}u\ln |u|$ for problem \eqref{eq} in future proofs,
and the methods used here are applicable to both \eqref{eg1} and \eqref{eg2}.
Moreover, the sufficient conditions for finite time blow-up of solutions to problem \eqref{eq} are weaker than those in \cite{Yang} and \cite{Zu}.

The rest of this paper is organized as follows. In Section 2,
we present some notations, definitions and lemmas that will be used in the sequel.
A new finite time blow-up criterion is given and an upper bound for the blow-up time
is derived in Section 3. In Section 4, the blow-up time is estimated from below when blow-up occurs.

\par
\section{Preliminaries}
\setcounter{equation}{0}

In this section, we introduce some notations and lemmas that will be used throughout the paper.
In what follows, we denote by $\|\cdot\|_2$ the $L^2(\Omega)$-norm,
by $(\cdot ,\cdot )$ the $L^2(\Omega)$-inner product and
by $\langle\cdot, \cdot\rangle_{*}$ the duality pairing between $H^{-1}(\Omega)$ and $H_0^1(\Omega)$.
We equip the Sobolev space $H_0^1(\Omega)$ with the inner product $\langle u,v \rangle=\int_\Omega (uv+\nabla u\cdot \nabla v) {\rm d}x$
and the norm $\|\cdot\|^2 = \|\cdot\|_2^2+\|\nabla\cdot\|_2^2$.
By $\lambda_1>0$ we denote the first eigenvalue of $-\Delta$ in $\Omega$ under homogeneous Dirichlet boundary condition.

The solution $u(x,t)$ to problem \eqref{eq} is considered in weak sense, whose definition is given below.
Sometimes $u(x,t)$ will be simply written as $u(t)$ if no confusion arises.
\begin{definition}
(See \cite{Gazzola}) We call $u(x,t)$ a weak solution to problem \eqref{eq} over $[0,T]$,
if $u\in C([0,T];H_0^1(\Omega))\cap C^1([0,T];L^2(\Omega))\cap C^2([0,T];H^{-1}(\Omega))$
with $u_t\in L^2(0,T;H_0^1(\Omega))$ such that $u(0)=u_0$, $u_t(0)=u_1$ and
\begin{equation}\label{weak}
\langle u_{tt}(t),\phi\rangle_{*}+\int_\Omega\nabla u(t)\cdot \nabla \phi{\rm d}x + \int_\Omega\nabla u_t(t)\cdot \nabla \phi{\rm d}x    +\int_\Omega u_t(t) \phi{\rm d}x=\int_\Omega  f(u(t))\phi{\rm d}x,
\end{equation}
for all $\phi\in H_0^1(\Omega)$ and a.e. $t\in[0,T]$.
\end{definition}

Since (H2) and (H3) hold, local existence and uniqueness of weak solutions to problem \eqref{eq}
can be established by slightly modifying the argument used in proving Theorem 3.1 in \cite{Gazzola},
and thus the details are omitted.

\begin{definition}\label{def-blow-up}
Let $u(t)$ be a weak solution to problem \eqref{eq}. The maximal existence time of $u(t)$ is defined by
$$T_{max}=\sup\{T>0:u=u(t)\  \text{exists on}\  [0,T]\}.$$
We say that $u(t)$ blows up at a finite time $T_{max}<+\infty$ provided that
\begin{equation*}
\lim_{t\rightarrow T_{max}} \int_0^t\|u(\tau)\|^2{\rm d}\tau +  \|u(t)\|_2^2=+\infty.
\end{equation*}
\end{definition}

We always associate problem \eqref{eq} with the energy functional $E(t)$
and Nehari's functional $I(u)$, which are defined, respectively, by
\begin{equation}\label{e}
E(t)=E(u(t))=\dfrac{1}{2}\|u_t\|_2^2+\dfrac{1}{2}\|\nabla u\|_2^2-\int_\Omega F(u){\rm d}x,
\quad  t\in[0,T_{max}),
\end{equation}
\begin{equation}\label{i}
I(u)=\|\nabla u\|_2^2-\int_\Omega f(u)u{\rm d}x, \quad u\in H_0^1(\Omega).
\end{equation}
By taking $u_t$ as a test function in \eqref{weak}, one obtains the following energy identity
 \begin{equation}\label{energy equality}
E(t)+ \int_0^t \|u_{\tau}\|^2 {\rm d}\tau=E(0), \quad a.e.\ t\in[0,T_{max}),
 \end{equation}
which implies that $E(t)$ is continuous and nonincreasing with respect to $t$ on $[0, T_{max})$.
Moreover, since $f$ satisfies (H2)-(H3), $I(u)$ is also well defined and continuous on $H_0^1(\Omega)$.
We define the unstable set by
\begin{equation}\label{n}
\mathcal{N}_{-}=\{u\in H_0^1(\Omega)\ |\  I(u)<0\}.
\end{equation}

We end up this section with the following concavity lemma, which will be needed to
prove the main results in Section 3.

\begin{lemma} (see \cite{Han, Levine,Levine1973})\label{3.3}
Suppose that a positive, twice-differentiable function $\psi (t)$ satisfies the inequality
\begin{eqnarray*}
\psi''(t)\psi(t)-(1+\theta)(\psi'(t))^2\geq0,
\end{eqnarray*}
where $\theta>0$. If $\psi(0)>0$ and $\psi'(0)>0$, then $\psi(t)\rightarrow\infty$ as
\begin{eqnarray*}
t\rightarrow t_*\leq t^*=\frac{\psi(0)}{\theta\psi'(0)}.
\end{eqnarray*}
\end{lemma}

\section{Blow-up and upper bound for the blow-up time}
\setcounter{equation}{0}

We begin this section with two lemmas which aim to prove that the unstable set $\mathcal{N}_{-}$
is invariant under the semi-flow of problem \eqref{eq}.
With the help of them, a new finite time blow-up criterion for problem \eqref{eq} can be established.

\begin{lemma}\label{3.1}
Let $u_0\in H_0^1(\Omega)$, $u_1\in L^2(\Omega)$ and $u=u(t)$ be a weak solution
to problem \eqref{eq} such that $u(t)\in \mathcal{N}_{-}$ on $[0,T_{max})$.
Then
\begin{equation*}
\{t\mapsto \|u(t)\|^2+2(u, u_t)\}
\end{equation*}
is strictly increasing on $(0,T_{max})$.
\end{lemma}

\begin{proof}
We define
\begin{equation}\label{F(t)}
K(t)= \|u(t)\|^2+2(u, u_t),\qquad t\in [0,T_{max}).
\end{equation}
By taking derivative and recalling the first equation in \eqref{eq}, we have
\begin{equation}\label{K'(t)}
\begin{split}
K'(t)=&2(u, u_t)+2(\nabla u, \nabla u_t)+2\|u_t\|_2^2+2(u, u_{tt})\\
      =&2\big[ (u, u_t)+(\nabla u, \nabla u_t)+\|u_t\|_2^2+(u, \Delta u+\Delta u_t-u_t+f(u))  \big]\\
      =&2\big[\|u_t\|_2^2-\|\nabla u\|_2^2+\int_\Omega u f(u){\rm d}x\big]\\
      =&2(\|u_t\|_2^2-I(u))>0,
\end{split}
\end{equation}
which implies that $K(t)$ is strictly increasing on $(0,T_{max})$.
Here we have used the assumption that $I(u(t))<0$ for $t\in[0,T_{max})$ to deduce the last inequality.
The proof is complete.
\end{proof}

\begin{lemma}\label{3.2}  (Invariance of $\mathcal{N}_{-}$)
Let $u_0\in H_0^1(\Omega)$ and $u_1\in L^2(\Omega)$. Assume that $u_0\in\mathcal{N}_{-}$ and the initial
data satisfy
\begin{equation}\label{E(0)}
\|u_0\|^2+2(u_0, u_1)>\frac{4p(1+\lambda_1)}{(p-2)\lambda_1}E(0).
\end{equation}
Then the solution $u(t)$ to problem \eqref{eq} belongs to $\mathcal{N}_{-}$  for all $t\in[0,T_{max})$.
\end{lemma}

\begin{proof}
We claim that $u(t)\in \mathcal{N}_{-}$ for all $t\in [0, T_{max})$.
Otherwise, by continuity, there would exist a $t_0 \in (0, T_{max})$ such that
\begin{equation}\label{3.2-1}
I(u(t)) < 0,\quad t\in [0, t_0),
\end{equation}
and
\begin{equation}\label{3.2-2}
I(u(t_0)) = 0.
\end{equation}
Then, it follows from Lemma \ref{3.1} that
$K(t)=\|u(t)\|^2+2(u, u_t)$ is strictly increasing on $[0, t_0)$.
Thanks to assumption \eqref{E(0)}, we have
\begin{equation}\label{3.2-3}
\|u(t)\|^2+2(u, u_t)>\|u_0\|^2+2(u_0, u_1)>\frac{4p(1+\lambda_1)}{(p-2)\lambda_1}E(0),\quad t\in (0,t_0).
\end{equation}
According to the continuity and monotonicity of $K(t)$, we arrive at
\begin{equation}\label{3.2-4}
\|u(t_0)\|^2+2(u(t_0), u_t(t_0))>\frac{4p(1+\lambda_1)}{(p-2)\lambda_1}E(0).
\end{equation}
On the other hand, from the definition of $E(t)$, $I(u)$ and the assumption (H1), we can derive
\begin{equation}\label{3.2-5}
\begin{split}
E(t)\geq & \frac{1}{2}\|u_t\|_2^2+\frac{1}{2}\|\nabla u\|_2^2-\frac{1}{p}\int_\Omega uf(u) {\rm d}x\\
      = & \frac{1}{2}\|u_t\|_2^2+\frac{p-2}{2p}\|\nabla u\|_2^2+\frac{1}{p} I(u),\quad t\in[0,T_{max}).
\end{split}
\end{equation}
In accordance with \eqref{energy equality}, \eqref{3.2-2}, \eqref{3.2-5} and Cauchy-Schwarz inequality, we know
\begin{equation}\label{3.2-6}
\begin{split}
  E(0)\geq & E(t_0)\\
     \geq& \frac{1}{2}\|u_t(t_0)\|_2^2+\frac{p-2}{2p}\|\nabla u(t_0)\|_2^2\\
     \geq& \frac{1}{2}\|u_t(t_0)\|_2^2+ \frac{(p-2)\lambda_1}{2p(1+\lambda_1)}
     \| u(t_0)\|_2^2+\frac{(p-2)\lambda_1}{2p(1+\lambda_1)}\|\nabla u(t_0)\|_2^2\\
     \geq& \frac{(p-2)\lambda_1}{2p(1+\lambda_1)}[\|u_t(t_0)\|_2^2+ \| u(t_0)\|_2^2+  \|\nabla u(t_0)\|_2^2 ] \\
     \geq& \frac{(p-2)\lambda_1}{4p(1+\lambda_1)} [ \|u_t(t_0)\|_2^2+ 2\| u(t_0)\|_2^2 + \|\nabla u(t_0)\|_2^2]\\
     \geq& \frac{(p-2)\lambda_1}{4p(1+\lambda_1)}[2(u_t(t_0), u(t_0)) + \| u(t_0)\|^2  ],
\end{split}
\end{equation}
which contradicts with \eqref{3.2-4}. The proof is complete.
\end{proof}


Next we show that the solutions to problem \eqref{eq} blow up in finite time
with initial energy that is bounded from above by $C_0\big(\|u_0\|^2+2(u_0,u_1)\big)$ for some $C_0>0$.
Moreover, an upper bound for the blow-up time is derived.

\begin{theorem}\label{blow-up}
Let all the assumptions in Lemma \ref{3.2} hold.
Then the solution $u(t)$ to problem \eqref{eq} blows up in finite time.
Moreover, $T_{max}$ can be estimated from above as follows
\begin{equation}\label{upper-bound}
 T_{max}\leq \dfrac{4\Big[\big(a^2+(\lambda-2)^2b_0\|u_0\|^2_2\big)^{1/2}+a\Big]}{(\lambda-2)^2b_0},
\end{equation}
where $\lambda,a,b_0$ are constants that will be determined in the proof.
\end{theorem}

\begin{proof}
This theorem will be proved by combining Levine's concavity argument with a specific choice of parameter.
Assume on the contrary that $u$ is a global weak solution to problem \eqref{eq}, then $T_{max}=\infty$.
For any $T >0$, $b>0$ and $\eta>0$, define

\begin{equation}\label{th3.1-1}
G(t)=\int_0^t \|u(\tau)\|^2 {\rm d}\tau +\|u(t)\|_2^2+ (T-t)\|u_0\|^2+ b(t+\eta)^2,\ \ t\in[0,T].
\end{equation}
Taking the first and second derivatives of the function $G(t)$, we have
\begin{equation}\label{th3.1-2}
\begin{split}
G'(t)=& \|u(t)\|^2+ 2(u,u_t)- \|u_0\|^2 +2b(t+\eta)  \\
=& 2(u,u_t)+2 \int_0^t \langle u,u_{\tau} \rangle {\rm d}\tau +  2b(t+\eta),\ \ t\in[0,T],
\end{split}
\end{equation}

\begin{equation}\label{th3.1-3}
\begin{split}
G''(t)=&  2(u_t, u_t)+2(u,u_{tt}) +2\langle u,u_t \rangle +2b \\
=& 2\|u_t\|_2^2 + 2(u, \Delta u+f(u)) + 2b \\
=& 2[\|u_t\|_2^2 -I(u)] + 2b,\ \ t\in[0,T].
\end{split}
\end{equation}
By \eqref{th3.1-2}, through a direct calculation, we obtain
\begin{equation}\label{th3.1-4}
\begin{split}
(G'(t))^2=& 4\Big[(u,u_t)+ \int_0^t \langle u,u_{\tau} \rangle {\rm d}\tau + b(t+\eta)\Big]^2 \\
=& 4\Big[ (u,u_t)^2+ 2(u,u_t)\int_0^t \langle u,u_{\tau} \rangle {\rm d}\tau+ 2b(u,u_t)(t+\eta) \\
&+ (\int_0^t \langle u,u_{\tau} \rangle {\rm d}\tau)^2 +2b(t+\eta)\int_0^t \langle u,u_{\tau} \rangle {\rm d}\tau+ b^2(t+\eta)^2\Big].
\end{split}
\end{equation}
According to Cauchy-Schwarz inequality, we know
\begin{equation*}
(u,u_t)\leq \|u\|_2 \|u_t\|_2,
\end{equation*}
\begin{equation*}
\int_0^t \langle u,u_{\tau} \rangle {\rm d}\tau\leq \int_0^t \|u\|\|u_{\tau}\| {\rm d}\tau
\leq\Big(\int_0^t \|u\|^2 {\rm d}\tau\Big)^{\frac{1}{2}}\Big(\int_0^t \|u_{\tau}\|^2 {\rm d}\tau\Big)^{\frac{1}{2}},
\end{equation*}
which, together with Cauchy's inequality, imply that
\begin{eqnarray}\label{th3.1-5}
&&\quad(G'(t))^2 \nonumber \\
&&\leq 4\Big[\|u\|_2^2\|u_t\|_2^2
+ 2\|u\|_2\|u_t\|_2\Big(\int_0^t \|u\|^2 {\rm d}\tau\Big)^{\frac{1}{2}}\Big(\int_0^t \|u_{\tau}\|^2 {\rm d}\tau\Big)^{\frac{1}{2}} \nonumber \\
&&\quad+ 2b(t+\eta)\|u\|_2\|u_t\|_2 +\int_0^t \|u\|^2 {\rm d}\tau\int_0^t \|u_{\tau}\|^2 {\rm d}\tau \nonumber \\
&&\quad+ 2b(t+\eta)\Big(\int_0^t \|u\|^2 {\rm d}\tau\Big)^{\frac{1}{2}}\Big(\int_0^t \|u_{\tau}\|^2 {\rm d}\tau\Big)^{\frac{1}{2}}
 +b^2(t+\eta)^2 \Big]  \\
&&\leq 4\Big[\|u\|_2^2\Big(\|u_t\|_2^2+\int_0^t \|u_{\tau}\|^2 {\rm d}\tau\Big)
+\int_0^t \|u\|^2 {\rm d}\tau\Big(\|u_t\|_2^2+\int_0^t \|u_{\tau}\|^2 {\rm d}\tau\Big)\nonumber \\
&&\quad + b\Big(\|u\|_2^2+\int_0^t \|u\|^2 {\rm d}\tau\Big)+b(t+\eta)^2\Big(\|u_t\|_2^2+\int_0^t \|u_{\tau}\|^2 {\rm d}\tau\Big)
+b^2(t+\eta)^2\Big]  \nonumber \\
&&=4\Big(\|u\|_2^2+\int_0^t \|u\|^2 {\rm d}\tau+ b(t+\eta)^2\Big)\Big(\|u_t\|_2^2+\int_0^t \|u_{\tau}\|^2 {\rm d}\tau+ b\Big).\nonumber
\end{eqnarray}
In view of \eqref{th3.1-1}, \eqref{th3.1-3} and \eqref{th3.1-5} we see that
\begin{equation}\label{th3.1-6}
\begin{split}
&G(t)G''(t)-\dfrac{\lambda+2}{4}(G'(t))^2\\
\geq& G(t)\Big[ G''(t)- (\lambda+2)(\|u_t\|_2^2+\int_0^t \|u_{\tau}\|^2 {\rm d}\tau +b) \Big]\\
=&G(t)\Big[-\lambda\|u_t\|_2^2-2I(u)-(\lambda+2)\int_0^t \|u_{\tau}\|^2 {\rm d}\tau-b\lambda \Big],
\end{split}
\end{equation}
where $2<\lambda<p$ will be decided later. Set
\begin{equation*}
H(t)=-\lambda\|u_t\|_2^2-2I(u)-(\lambda+2)\int_0^t \|u_{\tau}\|^2 {\rm d}\tau-b\lambda.
\end{equation*}
Following from \eqref{energy equality}, \eqref{i} and the assumption (H1), we have
\begin{equation}\label{th3.1-7}
\begin{split}
H(t)\geq &(p-\lambda)\|u_t\|_2^2+(p-2)\|\nabla u\|_2^2-2pE(0)+(2p-\lambda-2)\int_0^t \|u_{\tau}\|^2 {\rm d}\tau-b\lambda\\
\geq & (p-\lambda)\|u_t\|_2^2+\Big[(p-2)-\frac{p-2}{1+\lambda_1}\Big]\|\nabla u\|_2^2+\frac{p-2}{1+\lambda_1}\|\nabla u\|_2^2-2pE(0)-b\lambda\\
\geq & (p-\lambda)\|u_t\|_2^2+ \frac{(p-2)\lambda_1}{1+\lambda_1}\| u\|^2-2pE(0)-b\lambda.
\end{split}
\end{equation}
For this moment we choose
\begin{equation}\label{lambda}
\lambda=p-\frac{(p-2)\lambda_1}{1+\lambda_1},
\end{equation}
which obviously guarantees that $\lambda\in(2,p)$.
By virtue of \eqref{th3.1-7} and Lemma \ref{3.1}, we see that
\begin{equation}\label{th3.1-8}
\begin{split}
H(t) \geq&  \frac{(p-2)\lambda_1}{1+\lambda_1}\Big(\|u_t\|_2^2+\| u\|^2\Big) -2pE(0)-b\lambda\\
\geq&  \frac{(p-2)\lambda_1}{2(1+\lambda_1)}\Big(2(u,u_t)+\| u\|^2\Big) -2pE(0)-b\lambda\\
\geq&  \frac{(p-2)\lambda_1}{2(1+\lambda_1)}\Big(2(u_0,u_1)+\| u_0\|^2- \frac{4p(1+\lambda_1)}{(p-2)\lambda_1}E(0)         \Big)-b\lambda.
\end{split}
\end{equation}
Moreover, since $u(t)$ is continuous with respect to $t$, it is not difficult to see that
\begin{equation}\label{th3.1-9}
G(t)\geq \xi >0, \quad \forall\ t\in [0,T],
\end{equation}
where $\xi$ is independent of $T$.
Hence, in accordance with \eqref{E(0)}, \eqref{th3.1-6},  \eqref{th3.1-8} and \eqref{th3.1-9}, we can deduce
\begin{equation}\label{th3.1-10}
G(t)G''(t)-\dfrac{\lambda+2}{4}(G'(t))^2 \geq 0,
\end{equation}
for any $t\in[0,T]$ and $b\in \Big(0, \frac{(p-2)\lambda_1}{2\lambda(1+\lambda_1)}\big(2(u_0,u_1)+\| u_0\|^2- \frac{4p(1+\lambda_1)}{(p-2)\lambda_1}E(0)\big)\Big]$.
Choose
\begin{equation}\label{th3.1-11}
\eta>\max\Big\{0,\frac{2\|u_0\|^2-(\lambda-2)(u_0,u_1)}{b(\lambda-2)}\Big\},
\end{equation}
which is also independent of $T$, then
$$G(0)=\|u_0\|_2^2+T\|u_0\|^2+b\eta^2>0,$$
$$G'(0)=2(u_0,u_1)+2b\eta>\frac{4\|u_0\|^2}{\lambda-2}>0,$$
and
\begin{equation}\label{th3.1-12}
\frac{4G(0)}{(\lambda-2)G'(0)}=\frac{2[\|u_0\|_2^2+T\|u_0\|^2+b\eta^2]}{(\lambda-2)[(u_0,u_1)+b\eta]}<T,
\end{equation}
for enough large $T$.
According to Lemma \ref{3.3}, there exists a $t_*>0$ satisfying
\begin{equation}\label{th3.1-13}
t_*\leq \frac{4G(0)}{(\lambda-2)G'(0)}(<T)
\end{equation}
such that
$$G(t)\rightarrow\infty\ \text{as}\ t\rightarrow t_*^-.$$
This contradicts with the assumption that $G(t)$ is well defined on $[0,T]$ for any $T>0$.
At this point, we have proved the finite time blow-up result of the solution.

Next, we will estimate the upper bound for the blow-up time. Note that
we denote the maximal existence time of $u(x,t)$ by $T_{max}$, which is finite by the above argument.
Then, for any $T\in (0, T_{max})$, we define $\bar{G}(t)$ similarly to the above $G(t)$ by
$$\bar{G}(t)=\int_0^t \|u(\tau)\|^2 {\rm d}\tau +\|u(\tau)\|_2^2+ (T_{max}-t)\|u_0\|^2+ b(t+\eta)^2,\ \ t\in[0,T].$$
According to the foregoing arguments, we can still obtain
$$T\leq \frac{2[\|u_0\|_2^2+T_{max}\|u_0\|^2+b\eta^2]}{(\lambda-2)[(u_0,u_1)+b\eta]},$$
where $\lambda$ is chosen as in \eqref{lambda},
$b\in \Big(0, \frac{(p-2)\lambda_1}{2\lambda(1+\lambda_1)}\big(2(u_0,u_1)+\| u_0\|^2- \frac{4p(1+\lambda_1)}{(p-2)\lambda_1}E(0)\big)\Big]$, and $\eta$ is still required to satisfy \eqref{th3.1-11}.
By the arbitrariness of $T< T_{max}$ it follows that
\begin{equation}\label{t-1}
T_{max}\leq \frac{2[\|u_0\|_2^2+T_{max}\|u_0\|^2+b\eta^2]}{(\lambda-2)[(u_0,u_1)+b\eta]},
\end{equation}
or equivalently,
\begin{equation}\label{t-2}
 T_{max}\leq T(\eta, b)\triangleq\frac{2(\|u_0\|^2_2+b\eta^2)}{(\lambda-2)[(u_0,u_1)+b\eta]-2\|u_0\|^2}.
\end{equation}
Fix a $b \in \Big(0, \frac{(p-2)\lambda_1}{2\lambda(1+\lambda_1)}\big(2(u_0,u_1)+\| u_0\|^2- \frac{4p(1+\lambda_1)}{(p-2)\lambda_1}E(0)\big)\Big]$.
Then, minimizing $T(\eta, b)$ for $\eta>\max\Big\{0,\frac{2\|u_0\|^2-(\lambda-2)(u_0,u_1)}{b(\lambda-2)}\Big\}$ one has
$$T_{min}(\eta, b)=T(\eta_0, b)=\dfrac{4\Big[\big(a^2+(\lambda-2)^2b\|u_0\|^2_2\big)^{1/2}+a\Big]}{(\lambda-2)^2b},$$
where $a=2\|u_0\|^2-(\lambda-2)(u_0,u_1)$ and $\eta_0=\dfrac{[a^2+(\lambda-2)^2b\|u_0\|_2^2]^{1/2}+a}{(\lambda-2)b}$.\\
Minimizing $T(\eta_0, b)$ for $b \in \Big(0, \frac{(p-2)\lambda_1}{2\lambda(1+\lambda_1)}\big(2(u_0,u_1)+\| u_0\|^2- \frac{4p(1+\lambda_1)}{(p-2)\lambda_1}E(0)\big)\Big]$ we finally obtain
$$T_{min}(\eta_0, b)=T(\eta_0,b_0)
=\dfrac{4\Big[\big(a^2+(\lambda-2)^2b_0\|u_0\|^2_2\big)^{1/2}+a\Big]}{(\lambda-2)^2b_0},$$
where $b_0=\frac{(p-2)\lambda_1}{2\lambda(1+\lambda_1)}\big(2(u_0,u_1)+\| u_0\|^2- \frac{4p(1+\lambda_1)}{(p-2)\lambda_1}E(0)\big)$.  \\

In conclusion,
$$T_{max}\leq \dfrac{4\Big[\big(a^2+(\lambda-2)^2b_0\|u_0\|^2_2\big)^{1/2}+a\Big]}{(\lambda-2)^2b_0}. $$
The proof of Theorem \ref{blow-up} is complete.
\end{proof}

\begin{remark}\label{rem1}
It is worth pointing out that for the solution $u(t)$ to problem \eqref{eq} to blow up in finite time,
the assumption \eqref{E(0)} is not required when the initial energy $E(0)<0$,
which can be verified by slightly modifying the proof of Theorem \ref{blow-up}.
Indeed, if we still define the auxiliary functional $G(t)$ as in \eqref{th3.1-1},
only with the exception that we choose $b\in(0, -2E(0)]$ and
$\eta>max\{0, \frac{2\|u_0\|^2-(p-2)(u_0, u_1)}{b(p-2)}\}$,
then similarly to the proof of Theorem \ref{blow-up} we deduce that
\begin{equation*}
\begin{split}
&G(t)G''(t)-\dfrac{p+2}{4}(G'(t))^2\\
\geq&G(t)\Big[ G''(t)-(p+2)(\|u_t\|_2^2+\int_0^t \|u_{\tau}\|^2 {\rm d}\tau +b) \Big]\\
\geq&G(t)\Big[(p-2)\|\nabla u\|_2^2+(p-2)\int_0^t \|u_{\tau}\|^2 {\rm d}\tau-2pE(0)-pb\Big]\\
\geq&0.
\end{split}
\end{equation*}
From this inequality and Lemma \ref{3.3},
it follows that the solution $u(t)$ to problem \eqref{eq} blows up in finite time.
\end{remark}

\begin{remark}\label{rem2}
For many cases of nonlinearities $f(u)$, Theorem \ref{blow-up}
implies that the solution to problem \eqref{eq} blows up in
finite time at arbitrarily high initial energy level.
For example, when $f(s)=|s|^{p-2}s$ with $p\in(2,2^*)$,
the corresponding initial energy functional $E(0)$ and the initial Nehari's functional $I(u_0)$
are defined, respectively, by
$$E(0)=\frac{1}{2}\|u_1\|_2^2+\frac{1}{2}\|\nabla u_0\|_2^2-\frac{1}{p}\|u_0\|_p^p,$$
$$I(u_0)=\|\nabla u_0\|_2^2-\|u_0\|_p^p.$$
For any $\overline{u}_0\in H_0^1(\Omega)$ and $\overline{u}_1\in L^2(\Omega)$
satisfying $(\overline{u}_0, \overline{u}_1)>0$, set $u_0=\alpha \overline{u}_0, u_1=\beta \overline{u}_1$,
where $\alpha, \beta$ will be given later.
It follows from $p>2$ that there exists an $\alpha_1>0$ such that
$$I(u_0)=I(\alpha\overline{u}_0)<0,\qquad \forall\ \alpha\geq\alpha_1.$$
Moreover, for any $H>0$, there exists an $\alpha_2>0$ such that
$$\|u_0\|^2=\|\alpha \overline{u}_0\|^2>\frac{4p(1+\lambda_1)}{(p-2)\lambda_1} H,\qquad\forall\ \alpha\geq\alpha_2.$$
Take $\alpha=\max\{\alpha_1, \alpha_2\}$.
For such an $\alpha>0$, there exists an appropriate $\beta>0$ such that
$$E(0)=H.$$
Recalling $(\overline{u}_0, \overline{u}_1)>0$, it is clear that
\begin{equation*}
\begin{split}
\|u_0\|^2+2(u_0, u_1) &> \frac{4p(1+\lambda_1)}{(p-2)\lambda_1} H\\
&=\frac{4p(1+\lambda_1)}{(p-2)\lambda_1} E(0).
\end{split}
\end{equation*}
According to Theorem \ref{blow-up} it is seen that the solution $u(x, t)$ to problem \eqref{eq} blows up in finite time with initial energy $E(0)=H$.

\end{remark}

\section{Lower bound for the blow-up time}
\setcounter{equation}{0}

Since the explicit blow-up time can seldom be obtained when blow-up occurs,
it is of great importance to estimate it from both above and below.
In this section, we aim to give an estimation of the blow-up time of solutions to problem
\eqref{eq} from below.
Throughout this section we shall use $C_1,C_2,\cdots$ to denote some positive constants
which may depend on $\Omega,q,n$, but are independent of the solution $u(x,t)$.

\begin{theorem}\label{lower-bound}
Let that $q$ satisfy (H2) and assume that $u(x,t)$ is a weak
solution to problem \eqref{eq} that blows up at $T_{max}$.
Then $$T_{max}\geq \int_{M(0)}^\infty\dfrac{{\rm d}s}{C_4+C_5s^{q}},$$
where $M(0)=\|u_1\|_2^2+\|\nabla u_0\|_2^2$ .
\end{theorem}

\begin{proof}
We only prove this theorem for $n\geq3$. The cases $n=1,2$ can be treated similarly.
As was done in \cite{Di}, define
\begin{equation}\label{4.1-1}
M(t)=\|u_t(t)\|_2^2+\|\nabla u(t)\|_2^2,\quad t\in[0,T_{max}).
\end{equation}
Then
\begin{equation}\label{4.1-2}
\lim_{t\rightarrow T_{max}}M(t)=+\infty.
\end{equation}
Taking the first derivative and recalling \eqref{eq} and the assumption $(H_2)$, we obtain
\begin{align}\label{4.1-3}
M'(t)&=2[(u_t,u_{tt})+(\nabla u,\nabla u_t)]\nonumber\\
&=2(u_t,u_{tt}-\Delta u)\nonumber\\
&=2(u_t,\Delta u_t-u_t+f(u))\nonumber\\
&\leq -2\|u_t\|^2+2\int_\Omega |u_t||f(u)|{\rm d}x\nonumber\\
&\leq -2\|u_t\|^2+ 2\alpha\int_\Omega |u_t|{\rm d}x+ 2\beta\int_\Omega |u_t||u|^q{\rm d}x.
\end{align}
Since $q\in(1, 2^{*}-1)$, it is not difficult to verify that $\frac{2nq}{n+2}<2^{*}$,
which implies that $H_0^1(\Omega)$ can be embedded into $L^{\frac{2nq}{n+2}}(\Omega)$ continuously.
Denote by $S_{r}$ the embedding constant from $H_0^1(\Omega)$ to $L^{r}(\Omega)$,
i.e.,
\begin{equation}\label{4.1-4}
\|v\|_{r}\leq S_{r}\|\nabla v\|_2,\qquad\forall\ v\in H_0^1(\Omega),
\end{equation}
where $r\in (1, \frac{2n}{n-2}]$ for $n\geq 3$.
Using H\"{o}lder's inequality, \eqref{4.1-4} and Young's inequality with $\varepsilon$,
we have
\begin{align}\label{4.1-5}
&2\alpha\int_\Omega |u_t|{\rm d}x+ 2\beta\int_\Omega |u_t||u|^q{\rm d}x\nonumber\\
\leq& 2\alpha |\Omega|^{\frac{n+2}{2n}}\|u_t\|_{\frac{2n}{n-2}}+ 2\beta\|u_t\|_{\frac{2n}{n-2}}\Big(\int_\Omega|u|^{\frac{2nq}{n+2}}{\rm d}x\Big)^{\frac{n+2}{2n}}\nonumber\\
\leq& 2\alpha |\Omega|^{\frac{n+2}{2n}}S_{\frac{2n}{n-2}}\|\nabla u_t\|_2
+ 2\beta S_{\frac{2n}{n-2}} S_{\frac{2nq}{n+2}}^q \|\nabla u_t\|_2\|\nabla u\|_2^q \\
\leq& \varepsilon\|\nabla u_t\|_2^2+C(\varepsilon)\Big[C_1+C_2\|\nabla u\|^{2q}_2\Big]\nonumber\\
\leq&\varepsilon \|u_t\|^2+C(\varepsilon)\Big[C_1+C_3M^{q}(t)\Big].\nonumber
\end{align}
Now we take $\varepsilon\leq 2$ and substitute \eqref{4.1-5} into \eqref{4.1-3} to obtain
\begin{equation}\label{4.1-6}
M'(t)\leq C_4+C_5M^{q}(t).
\end{equation}
Integrating the inequality \eqref{4.1-6} over $[0,t]$, we arrive at
\begin{equation}\label{4.1-7}
\int_0^t\dfrac{M'(\tau)}{C_4+C_5M^{q}(\tau)}{\rm d}\tau\leq t.
\end{equation}
Letting $t\rightarrow T_{max}$ and recalling \eqref{4.1-2}, we obtain
\begin{equation}\label{4.1-8}
\int_{M(0)}^\infty\dfrac{{\rm d}s}{C_4+C_5s^{q}}\leq T_{max}.
\end{equation}
Since $q>1$, the left-hand side term in \eqref{4.1-8} is finite.
The proof is complete.
\end{proof}

\begin{remark}\label{rem3}
It is directly verified that the lower bound derived in Theorem \ref{lower-bound}
is applicable to the case when $f(s)=|s|^{\gamma-2}s\ln|s|$ with $\gamma\in(2,2^*)$.
It should be noticed that by making full use of the strong damping term,
we obtain the lower bound not only for subcritical exponent $\gamma$, i.e., $2<\gamma<\frac{2n-2}{n-2}$,
but also for the supercritical case $\frac{2n-2}{n-2}<\gamma<2^*$.
This partially extends the corresponding results obtained in \cite{Di}.
\end{remark}

{\bf Acknowledgements}\\
The authors would like to express their sincere gratitude to 
Professor Wenjie Gao in Jilin University for his enthusiastic
guidance and constant encouragement and to Professor Bin Guo 
for some valuable discussions when proving Theorem \ref{blow-up}.

\end{document}